\documentclass[11pt]{amsart}
\usepackage{mathrsfs,latexsym,amsfonts,amssymb}
\newtheorem{theorem}{Theorem}[section]
\newtheorem{lemma}[theorem]{Lemma}
\newtheorem{corollary}[theorem]{Corollary}
\newtheorem{question}[theorem]{Question}
\newtheorem{example}[theorem]{Example}
\theoremstyle{definition}
\newtheorem{definition}[theorem]{Definition}
\newtheorem{proposition}[theorem]{Proposition}
\theoremstyle{remark}

\begin{document}

\title[Feathered gyrogroups and gyrogroups with countable pseudocharacter]
{Feathered gyrogroups and gyrogroups with countable pseudocharacter}

\author{Meng Bao}
\address{(Meng Bao): School of mathematics and statistics,
Minnan Normal University, Zhangzhou 363000, P. R. China}
\email{mengbao95213@163.com}

\author{Fucai Lin*}
\address{(Fucai Lin): School of mathematics and statistics,
Minnan Normal University, Zhangzhou 363000, P. R. China}
\email{linfucai2008@aliyun.com; linfucai@mnnu.edu.cn}

\thanks{The authors are supported by the NSFC (No. 11571158), the Natural Science Foundation of Fujian Province (No. 2017J01405) of China, the Program for New Century Excellent Talents in Fujian Province University, the Institute of Meteorological Big Data-Digital Fujian and Fujian Key Laboratory of Data Science and Statistics.\\
*corresponding author}

%insert keywords
%insert subject class

%\date{\today}
\begin{abstract}
Topological gyrogroups, with a weaker algebraic structure than groups, have been investigated recently. In this paper, we prove that every feathered strongly topological gyrogroup is paracompact, which implies that every feathered strongly topological gyrogroup is a $D$-space and gives partial answers to two questions posed by A.V.Arhangel' ski\v\i ~(2010) in \cite{AA1}. Moreover, we prove that every locally compact $NSS$-gyrogroup is first-countable. Finally, we prove that each Lindel\"{o}f $P$-gyrogroup is Ra$\check{\imath}$kov complete.
\end{abstract}

\maketitle

\section{Introduction and Preliminaries}
Recently, W. Atiponrat \cite{AW} has defined the concept of topological gyrogroups as a generalization of topological groups and discussed some properties of them. In \cite{CZ}, the authors proved that each topological gyrogroup is a rectifiable space. Therefore, M\"{o}bius gyrogroup and Einstein gyrogroups are all rectifiable spaces, see \cite{AW}. A gyrogroup is a relaxation of a group such that the associativity condition has been replaced by a weaker one, which has become an interesting topic in algebra, see \cite{FM,FM1,SL,ST,ST1}. Indeed, this concept of a gyrogroup is originated from the research of $c$-ball of relativistically admissible velocities with Einstein velocity addition as mentioned by A.A. Ungar in \cite{UA}.

The purpose of this paper is to investigate some properties of topological gyrogroups and the coset space of topological gyrogroups. In particular, we prove that every feathered strongly topological gyrogroup is paracompact, every locally compact $NSS$-gyrogroup $G$ is first-countable and each Lindel\"{o}f $P$-gyrogroup $G$ is Ra$\check{\imath}$kov complete. These results extend some well known results of topological groups.

Throughout this paper, all topological spaces are assumed to be
$T_{1}$, unless otherwise is explicitly stated. Let $\mathbb{N}$ be the set of all positive integers and $\omega$ the first infinite ordinal. The readers may consult \cite{AA, E, linbook} for notations and terminologies not explicitly given here. Next we recall some definitions and facts.

\begin{definition}\cite{AW}
Let $G$ be a nonempty set, and let $\oplus: G\times G\rightarrow G$ be a binary operation on $G$. Then the pair $(G, \oplus)$ is called a {\it groupoid}. A function $f$ from a groupoid $(G_{1}, \oplus_{1})$ to a groupoid $(G_{2}, \oplus_{2})$ is called a {\it groupoid homomorphism} if $f(x\oplus_{1}y)=f(x)\oplus_{2} f(y)$ for any elements $x, y\in G_{1}$. Furthermore, a bijective groupoid homomorphism from a groupoid $(G, \oplus)$ to itself will be called a {\it groupoid automorphism}. We write $\mbox{Aut}(G, \oplus)$ for the set of all automorphisms of a groupoid $(G, \oplus)$.
\end{definition}

\begin{definition}\cite{UA}
Let $(G, \oplus)$ be a groupoid. The system $(G,\oplus)$ is called a {\it gyrogroup}, if its binary operation satisfies the following conditions:

\smallskip
$(G1)$ There exists a unique identity element $0\in G$ such that $0\oplus a=a=a\oplus0$ for all $a\in G$.

\smallskip
$(G2)$ For each $x\in G$, there exists a unique inverse element $\ominus x\in G$ such that $\ominus x \oplus x=0=x\oplus (\ominus x)$.

\smallskip
$(G3)$ For all $x, y\in G$, there exists $gyr[x, y]\in \mbox{Aut}(G, \oplus)$ with the property that $x\oplus (y\oplus z)=(x\oplus y)\oplus gyr[x, y](z)$ for all $z\in G$.

\smallskip
$(G4)$ For any $x, y\in G$, $gyr[x\oplus y, y]=gyr[x, y]$.
\end{definition}

Notice that a group is a gyrogroup $(G,\oplus)$ such that $gyr[x,y]$ is the identity function for all $x, y\in G$. The definition of a subgyrogroup is as follows.

\begin{definition}\cite{ST}
Let $(G,\oplus)$ be a gyrogroup. A nonempty subset $H$ of $G$ is called a {\it subgyrogroup}, denoted
by $H\leq G$, if the following statements hold:

\smallskip
$(i)$ The restriction $\oplus| _{H\times H}$ is a binary operation on $H$, i.e. $(H, \oplus| _{H\times H})$ is a groupoid.

\smallskip
$(ii)$ For any $x, y\in H$, the restriction of $gyr[x, y]$ to $H$, $gyr[x, y]|_{H}$ : $H\rightarrow gyr[x, y](H)$, is a bijective homomorphism.

\smallskip
$(iii)$ $ (H, \oplus|_{H\times H})$ is a gyrogroup.

\smallskip
Furthermore, a subgyrogroup $H$ of $G$ is said to be an {\it $L$-subgyrogroup} \cite{ST}, denoted
by $H\leq_{L} G$, if $gyr[a, h](H)=H$ for all $a\in G$ and $h\in H$.

\end{definition}

\begin{proposition}\cite{ST}.
Let $(G, \oplus)$ be a gyrogroup, and let $H$ be a nonempty subset of $G$. Then $H$ is a subgyrogroup if and only if the following statements are true:

\smallskip
1. For any $x\in H$, $\ominus x\in H$.

\smallskip
2. For any $x, y\in H$, $x\oplus y\in H$.
\end{proposition}

\begin{lemma}\cite{UA}\label{a}.
Let $(G, \oplus)$ be a gyrogroup. Then for any $x, y, z\in G$, we obtain the following:

\smallskip
1. $(\ominus x)\oplus (x\oplus y)=y$.

\smallskip
2. $(x\oplus (\ominus y))\oplus gyr[x, \ominus y](y)=x$.

\smallskip
3. $(x\oplus gyr[x, y](\ominus y))\oplus y=x$.

\smallskip
4. $gyr[x, y](z)=\ominus (x\oplus y)\oplus (x\oplus (y\oplus z))$.
\end{lemma}

\begin{definition}\cite{AW}
A triple $(G, \tau, \oplus)$ is called a {\it topological gyrogroup} if the following statements hold:

\smallskip
1. $(G, \tau)$ is a topological space.

\smallskip
2. $(G, \oplus)$ is a gyrogroup.

\smallskip
3. The binary operation $\oplus: G\times G\rightarrow G$ is jointly continuous while $G\times G$ is endowed with the product topology, and the operation of taking the inverse $\ominus (\cdot): G\rightarrow G$, i.e. $x\rightarrow \ominus x$, is also continuous.
\end{definition}

\begin{definition}\cite{ST1}
Let $(G, \tau, \oplus)$ be a topological gyrogroup, and let $A, B$ be subsets of $G$. We will define $A\oplus B$ and $\ominus A$ as follows:

$A\oplus B= \{a\oplus b: a\in A, b\in B\}$ and
$\ominus A= \{\ominus a: a\in A\}=(\ominus)^{-1}(A)$.
\end{definition}

\begin{definition}\cite{UA}
Let $(G,\oplus)$ be a gyrogroup, and let $x\in G$. We define the {\it left gyrotranslation} by $x$ to be the function $L_{x} : G\rightarrow G$ such that $L_{x} (y)=x\oplus y$ for any $y\in G$. In addition, the {\it right gyrotranslation} by $x$ is defined to be the function $R_{x} (y)=y\oplus x$ for any $y\in G$.
\end{definition}

\begin{definition}\cite{E}
A topological space $X$ is {\it paracompact} if every open covering of $X$ can be refined by a locally finite open covering.
\end{definition}

\begin{definition}\cite{E}
A continuous mapping $f:X\rightarrow Y$ is called {\it closed} (resp. {\it open}) if for every closed (resp. open) set $A\subset X$ the image $f(A)$ is closed (open) in $Y$.
\end{definition}

\begin{definition}\cite{E}
A continuous mapping $f:X\rightarrow Y$ is {\it perfect} if $f$ is a closed mapping and all fibers $f^{-1}(y)$ are compact subsets of $X$.
\end{definition}

\begin{definition}\cite{E}
A family $\gamma$ of open sets in a space $X$ is called a base for $X$ at a set $F\subset X$ if all elements of $\gamma$ contains $F$ and, for each open set $V$ that contains $F$, there exists $U\in\gamma$ such that $U\subset V$. The character of $X$ at a set $F$ is the smallest cardinality of a base for $X$ at $F$.
\end{definition}

\section{Basic Properties of topological gyrogroups}
In this section, we study some basic properties of topological gyrogroups. First, we recall two lemmas.

\begin{lemma}\cite{AW}\label{n}.
Let $(G,\tau ,\oplus)$ be a topological gyrogroup, and let $x\in G$. Then $$\ominus (\cdot) : (G,\tau)\rightarrow (G,\tau), L_{x} : (G,\tau)\rightarrow (G,\tau)\ \mbox{and}\ R_{x}: (G,\tau)\rightarrow (G,\tau)$$ are all homeomorphisms.
\end{lemma}

\begin{lemma}\cite{AW}\label{m}
Let $(G, \tau, \oplus)$ be a topological gyrogroup, and let $U$ be a neighborhood of the identity $0$. Then the following three statements hold:
\begin{enumerate}
\item There is an open symmetric neighborhood $V$ of $0$ in $G$ such that $V\subset U$ and $V\oplus V\subset U$.

\smallskip
\item There is an open neighborhood $V$ of $0$ such that $\ominus V\subset U$.

\smallskip
\item If $A$ is a subset of $G$, $\overline{A}\subset W\oplus A$ for any neighborhood $W$ of the identity $0$.
\end{enumerate}
\end{lemma}

\begin{proposition}\label{p}\cite{AW}
Let $(G,\tau ,\oplus)$ be a topological gyrogroup, and $H$ a subgyrogroup of $G$. Then the closure of $H$ in $G$ is also a subgyrogroup of $G$.
\end{proposition}

\begin{proposition}\cite{AW}
Let $(G,\oplus)$ be a gyrogroup, $U$ an open subset of $G$ and $A$ any subset of $G$. Then the set $U\oplus A$ and $A\oplus U$ are all open in $G$.
\end{proposition}

\begin{lemma}\label{o}
For any subsets $A, B$ and $C$ of a gyrogroup $G$, we have $(A\oplus B)\bigcap C=\emptyset$ if and only if $B\bigcap ((\ominus A)\oplus C)=\emptyset$.
\end{lemma}

\begin{proof}
Obviously, $B\bigcap ((\ominus A)\oplus C)\neq\emptyset$ if and only if there exists a point $x\in G$ such that $x\in B\bigcap ((\ominus A)\oplus C)$, if and only if there exist $a\in A, b\in B, c\in C$ such that $x=b=\ominus a\oplus c$, if and only if $a\oplus b=c$, if and only if $(A\oplus B)\bigcap C\neq\emptyset$.
\end{proof}

\begin{theorem}\label{t29}
Let $(G, \tau, \oplus)$ be a topological gyrogroup, and $\mathscr U$ a neighborhood base of the space $G$ at the identity element $0$. Then, for every subset $A$ of $G$, we have $$\overline{A}=\bigcap  \{{U\oplus A, U\in \mathscr U}\}.$$
\end{theorem}

\begin{proof}

By (3) in Lemma~\ref{m}, it is clear that $\overline{A}\subset \bigcap \{U\oplus A, U\in \mathscr U\}$.

Next, it suffices to prove that $\bigcap \{U\oplus A,U\in \mathscr U\}\subset \overline{A}$. Assume $x\not\in \overline{A}$, then there exists $W\in \mathscr U$ such that $(W\oplus x)\bigcap A=\emptyset$. Let $U\in \mathscr U$ such that $\ominus U=U\subset W$. Then $((\ominus U)\oplus x)\bigcap A=\emptyset$. By Lemma \ref{o}, we have $x\not\in U\oplus A$. Hence, $\bigcap \{U\oplus A,U\in \mathscr U\}\subset \overline{A}$.

\end{proof}

\begin{theorem}\label{ttt}
Let $(G, \tau, \oplus)$ be a topological gyrogroup, and $\mathscr U$ a family of neighborhoods of the identity $0$ satisfying:

\smallskip
$a)$ For every $U\in \mathscr U$, there exists $V\in \mathscr U$ such that $V\oplus V\subset U$.

\smallskip
$b)$ For every $U\in \mathscr U$, there exists $V\in \mathscr U$ such that $\ominus V\subset U$.

\smallskip
Then $\bigcap \mathscr U$ is a closed subgyrogroup of $G$.

\end{theorem}

\begin{proof}
Let $H=\bigcap \mathscr U$. We assume that conditions $a)$ and $b)$ are satisfied. It is easy to see that $H=\bigcap \{U\oplus U: U\in \mathscr U\}=\bigcap \{\ominus U: U\in \mathscr U\}$. Clearly, for every $U\in \mathscr U$, we have that $H\oplus H\subset U\oplus U$ and $\ominus H\subset \ominus U$. Hence it follows that $H$ is a subgyrogroup of $G$. By (3) in Lemma~\ref{m}, we have that $\overline {U}\subset U\oplus U$ for every $U\in \mathscr U$. Therefore, $$\overline {H}=\overline{\bigcap \{\mathscr U\}}\subset \bigcap \{\overline{U}: U\in \mathscr U\}\subset \bigcap \{U\oplus U:U\in \mathscr U\}=H.$$ Thus $H=\overline{H}$.
\end{proof}

By \cite{CZ}, we see that every topological gyrogroup $G$ is rectifiable. Therefore, it follows from \cite{CZ} and \cite{LF2} respectively that we have the following two corollaries.

\begin{corollary}\label{d}\cite{CZ}
If  $(G,\tau ,\oplus)$ is a topological gyrogroup, then the gyrogroup is first-countable if and only if it is metrizable.
\end{corollary}

\begin{corollary}\label{c}\cite{LF2}
Let $(G,\tau ,\oplus)$ be a topological gyrogroup. Suppose that $H$ is a compact subgyrogroup of $G$, and $S$ is a closed subgyrogroup of $G$. Then $H\oplus S$ and $S\oplus H$ are all closed in G.
\end{corollary}

\smallskip
\section{feathered topological gyrogroups}
In this section, we discuss the strongly topological gyrogroups. We mainly prove that each feathered strongly topological gyrogroup is paracompact. First, we define the strongly topological gyrogroup.

Let $G$ be a topological gyrogroup. We say that $G$ is a {\it strongly topological gyrogroup} if there exists a neighborhood base $\mathscr U$ of $0$ such that, for every $U\in \mathscr U$, $gyr[x,y](U)=U$ for any $x, y\in G$. For convenience, we say that $G$ is a strongly topological gyrogroup with neighborhood base $\mathscr U$ of $0$.
Clearly, we may assume that $U$ is symmetric for each $U\in\mathscr U$. Moreover, it is easy to see that each topological group is a strongly topological gyrogroup. However, there exists a strongly topological gyrogroup which is not a topological group, see Example~\ref{e0}.

\begin{example}\label{e0}
There exists a strongly topological gyrogroup which is not a topological group.
\end{example}

\smallskip
Indeed, let $D$ be the complex open unit disk $\{z\in C:|z|<1\}$. We consider $D$ with the standard topology. In \cite[Example 2]{AW}, define a M\"{o}bius addition $\oplus _{M}: D\times D\rightarrow D$ to be a function such that $$a\oplus _{M}b=\frac{a+b}{1+\bar{a}b}\ \mbox{for all}\ a, b\in D.$$ Then $(D, \oplus _{M})$ is a gyrogroup, and it follows from \cite[Example 2]{AW} that $$gyr[a, b](c)=\frac{1+a\bar{b}}{1+\bar{a}b}c\ \mbox{for any}\ a, b, c\in D.$$ For any $n\in\omega$, let $U_{n}=\{x\in D: |x|\leq \frac{1}{n}\}$. Then, $\mathscr U=\{U_{n}: n\in \omega\}$ is a neighborhood base of $0$. Moreover, we observe that $|\frac{1+a\bar{b}}{1+\bar{a}b}|=1$. Therefore, we obtain that $gyr[x, y](U)\subset U$, for any $x, y\in D$ and each $U\in \mathscr U$, then it follows that $gyr[x, y](U)=U$ by \cite[Proposition 2.6]{ST}. Hence, $(D, \oplus _{M})$ is a strongly topological gyrogroup. However, $(D, \oplus _{M})$ is not a group \cite[Example 2]{AW}.

\bigskip
{\bf Remark} M\"{o}bius gyrogroups, Einstein gyrogroups, and Proper velocity gyrogroups, that were studied e.g. in \cite{FM, FM1,UA},  are all strongly topological gyrogroups; however, they do not possess any non-trivial $L$-subgyrogroups, hence the theory of this paper does't apply to the M\"{o}bius gyrogroups, Einstein gyrogroups or Proper Velocity gyrogroups. However, there exists a class of strongly topological gyrogroups which has a non-trivial $L$-subgyrogroup, see Example~\ref{e1}.

\begin{example}\label{e1}
There exists a strongly topological gyrogroup which has an infinite $L$-subgyrogroup.
\end{example}

\smallskip
Indeed, let $X$ be an arbitrary feathered non-metrizable topological group, and let $Y$ be an any strongly topological gyrogroup with a non-trivial $L$-subgyrogroup (such as the gyrogroup $K_{16}$ \cite[p. 41]{UA2002}). Put $G=X\times Y$ with the product topology and the operation with coordinate. Then $G$ is an infinite strongly topological gyrogroup since $X$ is infinite. Let $H$ be a non-trivial $L$-subgyrogroup of $Y$, and take an arbitrary infinite subgroup $N$ of $X$. Then $N\times H$ is an infinite $L$-subgyrogroup of $G$.

\bigskip
The following two questions are well known in the study of rectifiable spaces.

\begin{question}\cite{AA1}\label{k}
Is every rectifiable $p$-space paracompact?
\end{question}

\begin{question}\cite{AA1}\label{j}
Is every rectifiable $p$-space a $D$-space?
\end{question}

It is well known that a paracompact $p$-space is a $D$-space. Obviously, if we can prove that each rectifiable p-space is paracompact, then the answer to Question \ref{j} is also positive. Since each topological gyrogroup is a rectifiable space, it is natural to pose the following two questions.

\begin{question}\label{kk}
If a topological gyrogroup is a $p$-space, is it paracompact? What if the topological gyrogroup is a strongly topological gyrogroup?
\end{question}

\begin{question}\label{jj}
If a topological gyrogroup is a $p$-space, is it a $D$-space? What if the topological gyrogroup is a strongly topological gyrogroup?
\end{question}

Next we prove that every feathered strongly topological gyrogroup is paracompact, which gives an affirmative answer to Question \ref{kk} when the topological gyrogroup is a strongly topological gyrogroup, see Corollary \ref{l}. Recall that a topological gyrogroup $G$ is {\it feathered} if it contains a non-empty compact set $K$ of countable character in $G$.  It is well known that each $p$-space is feathered.

We recall the following concept of the coset space of a topological gyrogroup.

Let $(G, \tau, \oplus)$ be a topological gyrogroup and $H$ a $L$-subgyrogroup of $G$. We define a binary operation on the coset $G/H$ in the following natural way:$$(a\oplus H)\oplus (b\oplus H)=(a\oplus b)\oplus H,$$ for any $a, b\in G$. By \cite[Theorem 20]{ST}, it follows that $G/H$ with the above operation forms a disjoint partition of $G$. We denote by $\varphi$ the mapping $a\mapsto a\oplus H$ from $G$ onto $G/H$. Clearly, we have $\varphi(a\oplus b)=\varphi(a)\oplus \varphi(b)$ for any $a, b\in G$, and for each $a\in G$, we have $\varphi^{-1}\{\varphi(a)\}=a\oplus H$.
Denote by $\tau (G)$ the topology of $G$. In the set $G/H$, we define a family $\tau (G/H)$ of subsets as follows: $$\tau (G/H)=\{O\subset G/H: \varphi^{-1}(O)\in \tau (G)\}.$$

\begin{theorem}\label{t00000}
Let $(G, \tau, \oplus)$ be a topological gyrogroup and $H$ a $L$-subgyrogroup of $G$. Then the natural homomorphism $\varphi$ from a topological gyrogroup $G$ to its quotient topology on $G/H$ is an open and continuous mapping.
\end{theorem}

\begin{proof}
By the definition of the quotient topology, the mapping $\varphi$ is continuous. Next we show that $\varphi$ is open. Then it suffices to prove that $\varphi(U)\in \tau(G/ H)$ for every $U\in \tau(G)$. By the definition of topology $\tau(G/ H)$, we have that $\varphi(U)\in \tau(G/ H)$ when $\varphi^{-1}(\varphi(U))\in \tau(G)$. For every $g\in G$, we have $\varphi^{-1}(\varphi(g))=g\oplus H$. As a consequence, $\varphi^{-1}(\varphi(U))=\bigcup_{g\in U}(g\oplus H) =U\oplus H$. So, $U\oplus H \in\tau(G)$. Hence, we have $\varphi(U)\in \tau(G/H)$ for every $U\in \tau(G)$. Therefore, $\varphi$ is an open mapping.
\end{proof}

\begin{theorem}\label{f}
Let $(G, \tau, \oplus)$ be a topological gyrogroup and $H$ a $L$-subgyrogroup of $G$. If $H$ is a compact subgyrogroup of $G$, then the quotient mapping $\varphi$ of $G$ onto the quotient space $G/H$ is perfect.
\end{theorem}

\begin{proof}
Take any closed subset $P$ of $G$. Since $H$ is a compact subgyrogroup of $G$, it follows from Corollary \ref{c} that $P\oplus H$ is closed in $G$. Obviously, $\varphi^{-1}(\varphi(P))=P\oplus H$. It follows from the definition of a quotient mapping that the set $\varphi(P)$ is closed in the quotient space $G/H$. Thus $\varphi$ is a closed mapping. Furthermore, if $y\in G/H$ and $\varphi(x)=y$ for some $x\in G$, we obtain that $\varphi^{-1}(y)=x\oplus H$ is a compact subset of $G$. Hence, the fibers of $\varphi$ are compact. Thus $\varphi$ is perfect.
\end{proof}

\begin{corollary}
Let $(G, \tau, \oplus)$ be a topological gyrogroup and $H$ a $L$-subgyrogroup of $G$. If $H$ is a compact subgyrogroup of $G$ such that the quotient space $G/H$ is compact, then $G$ is compact.
\end{corollary}

\begin{lemma}\label{b}
Let $(G, \tau, \oplus)$ be a topological gyrogroup with identity element $0$ and $F$ a compact subset of $G$ containing $0$ and having a countable base $\{U_{n}:n\in \omega\}$ in $G$. Suppose that a sequence $\gamma =\{V_{n}: n\in \omega\}$ of open symmetric neighborhoods of $0$ in $G$ satisfies $V_{n+1}\oplus V_{n+1}\subset V_{n}\cap U_{n}$ for each $n\in \omega$. Then $H=\bigcap _{n\in \omega}V_{n}$ is a compact subgyrogroup of $G$, $H\subset F$, and $\gamma$ is a base for $G$ at $H$. In particular, if $(G, \tau, \oplus)$ is a strongly topological gyrogroup with a base $\mathscr{U}$ at $0$ consists of symmetric neighborhoods such that each $V_{n}\in \mathscr{U}$, then $H$ is also a $L$-subgyrogroup of $G$.
\end{lemma}

\begin{proof}
For each $n\in \omega$, since $V_{n+1}\oplus V_{n+1}\subset V_{n}$, we have $\overline{V_{n+1}} \subset V_{n}$. From the definition of $H$, it follows that $H=\bigcap _{n\in \omega}V_{n}=\bigcap _{n\in \omega}\overline{V_{n+1}}$ is closed $L$-subgyrogroup of $G$. In addition, $H\subset \bigcap _{n\in \omega}U_{n}=F$, so $H$ is compact.

Let $W$ be an open neighborhood of $H$ in $G$. Then $K=F\backslash W$ is a compact subset of $G$ disjoint from $H$. If $\overline{V_{n}}\bigcap K\not=\emptyset$ for each $n\in \omega$, then, by the compactness of $K$, the set $K\cap (\bigcap _{n\in \omega}V_{n})=K\cap H$ is not empty, which is a contradiction. So, $K\cap V_{n}=\emptyset$ for some $n\in \omega$. It is clear that $F\subset W\cup K\subset W\cup(V_{n+1}\oplus K)$. Since $W\cup (V_{n+1}\oplus K)$ is open in $G$, there exists an integer $m$ such that $U_{m}\subset W\cup (V_{n+1}\oplus K)$. Let $k=max\{m,n\}$. Observe that $K\cap (V_{k+1}\oplus V_{n+1})\subset K\cap V_{n}=\emptyset$. We claim that $(V_{n+1}\oplus K)\cap V_{k+1}=\emptyset$. Suppose not, assume that $(V_{n+1}\oplus K)\cap V_{k+1}\not=\emptyset$. Then there exists $v\in V_{n+1}$, $h\in K$, $u\in V_{k+1}$ such that $v\oplus h=u$. By Theorem \ref{a}, we have that $h=\ominus v\oplus u\in V_{n+1}\oplus V_{k+1}$. Therefore, $K\cap (V_{n+1}\oplus V_{k+1})\not=\emptyset$, which is contradiction with $K\cap V_{n}=\emptyset$. Therefore, $V_{k+1}=V_{k+1}\cap U_{m}\subset V_{k+1}\cap (W\cup (V_{n+1}\oplus K))=V_{k+1}\cap W\subset W$. That is $V_{k+1}\subset W$. This proves that $\gamma$ is a base for $G$ at $H$.

Clearly, if $(G, \tau, \oplus)$ is a strongly topological gyrogroup, then $$gyr[x, y](H)=gyr[x, y](\bigcap _{n\in \omega}V_{n})\subset\bigcap _{n\in \omega}gyr[x, y](V_{n})=\bigcap _{n\in \omega}V_{n}=H,$$ hence $H$ is a $L$-subgyrogroup of $G$.
\end{proof}

\begin{theorem}\label{g}
Let $(G, \tau, \oplus)$ be a feathered strongly topological gyrogroup and $O$ be a neighborhood of the identity element $0$ in $G$. Then there exists a compact $L$-subgyrogroup $H$ of countable character in $G$ satisfying $H\subset O$.
\end{theorem}

\begin{proof}
Let $\mathscr{U}$ be the symmetric neighborhood base at $0$ such that for any $x, y\in G$, $gyr[x, y](U)=U$ for any $U\in\mathscr{U}$. Since the gyrogroup $G$ is feathered, it contains a non-empty compact set $F$ with $\chi (F, G)\leq \omega$. By the homogeneity of $G$, we can assume that $F$ contains the identity element $0$ of $G$. Let $\{U_{n}:n\in \omega \}$ be a countable base for $G$ at $F$. We define by induction a sequence $\{V_{n}:n\in \omega\}$ of symmetric open neighborhood of $0$ in $G$ satisfying the following conditions:

\smallskip
(i) $V_{0}\subset O$.

\smallskip
(ii) $V_{n+1}\oplus V_{n+1}\subset U_{n}\bigcap V_{n}$ for each $n\in \omega$.

\smallskip
(iii) $V_{n}\in\mathscr{U}$ for each $n\in \omega$.

\smallskip
Put $H=\bigcap _{n\in \omega}V_{n}$. Then $H\subset O$ by (i). Furthermore, it follows from Lemma~\ref{b} that $H$ is a compact $L$-subgyrogroup of $G$ and  $\{V_{n}:n\in \omega\}$ is a base for $G$ at $H$.
\end{proof}

\begin{lemma}\label{s}
Let $G$ be a strongly topological gyrogroup with the symmetric neighborhood base $\mathscr{U}$ at $0$, and let $\{U_{n}: n\in\omega\}$ and $\{V(m/2^{n}): n, m\in\omega\}$ be two sequences of open neighborhoods satisfying the following conditions (1)-(5):

\smallskip
(1) $U_{n}\in\mathscr{U}$ for each $n\in \omega$.

\smallskip
(2) $U_{n+1}\oplus U_{n+1}\subset U_{n}$, for each $n\in \omega$.

\smallskip
(3) $V(1)=U_{0}$;

\smallskip
(4) For any $n\geq 1$, put $$V(1/2^{n})=U_{n}, V(2m/2^{n})=V(m/2^{n-1})$$ for $m=1,...,2^{n}$, and $$V((2m+1)/2^{n})=U_{n}\oplus V(m/2^{n-1})=V(1/2^{n})\oplus V(m/2^{n-1})$$ for each $m=1,...,2^{n}-1$;

\smallskip
(5) $V(m/2^{n})=G$ when $m>2^{n}$;

\smallskip
Then there exists a prenorm $N$ on $G$ satisfies the following conditions:

\smallskip
(a) for any fixed $x, y\in G$, we have $N(gyr[x,y](z))=N(z)$ for any $z\in G$;

\smallskip
(b) for any $n\in\omega$, $$\{x\in G:N(x)<1/2^{n}\}\subset U_{n}\subset\{x\in G:N(x)\leq 2/2^{n}\}.$$
\end{lemma}

\begin{proof}
 Firstly, we claim that $$V(1/2^{n})\oplus V(m/2^{n})\subset V((m+1)/2^{n}),$$ for all integers $m>0$ and $n\geq 0$. It is obviously true if $m+1>2^{n}$. It remains to consider the case when $m<2^{n}$.

If $n=1$, then the only possible value for $m$ is also $1$, and we have $$V(1/2)\oplus V(1/2)=U_{1}\oplus U_{1}\subset U_{0}=V(1).$$

Assume that the condition holds for some $n\in\omega$. Let us verify it for $n+1$. If $m$ is even, then the condition turns into the formula by means of which $V((2m+1)/2^{n+1})$ was defined.

Assume now that $0<m=2k+1<2^{n+1}$, for some integer $k$. Therefore,
\begin{eqnarray}
V(1/2^{n+1})\oplus V(m/2^{n+1})&=&U_{n+1}\oplus V((2k+1)/2^{n+1})\nonumber\\
&=&U_{n+1}\oplus (U_{n+1}\oplus V(k/2^{n}))\nonumber\\
&=&(U_{n+1}\oplus U_{n+1})\oplus gyr[U_{n+1},U_{n+1}](V(k/2^{n}))\nonumber\\
&\subset&U_{n}\oplus V(k/2^{n})\nonumber\\
&=&V(1/2^{n})\oplus V(k/2^{n}).\nonumber
\end{eqnarray}
By the inductive assumption, we have $$V(1/2^{n})\oplus V(k/2^{n})\subset V((k+1)/2^{n})=V((2k+2)/2^{n+1})=V((m+1)/2^{n+1}).$$

Now we define a real-valued function $f$ on $G$ as follows  $$f(x)=\inf\{r>0:x\in V(r)\}$$ for each $x\in G$. The function $f$ is well-defined, since $x\in V(2)=G$, for each $x\in G$. From the claim above it follows that if $0<r<s$ for positive dyadic rational numbers $r$ and $s$, then $V(r)\subset V(s)$. Moreover, in the argument below, stand only for positive dyadic rational numbers. Hence we have that $x\in V(r)$ whenever $f(x)<r$. Then $f$ is a non-negative function, bounded from above by $2$.

Next, we will show that the function $N$ defined by the formula $$N(x)=sup_{y\in G}|f(x\oplus y)-f(y)|$$ for any $x\in G$, is a prenorm on $G$. We first claim that $N(gyr[x, y](z))=N(z)$, for any $x, y, z\in G$.

Indeed, it need to verify that for any fixed $x, y\in G$ we have, for every $z\in G$, $z\in V(r)$ if and only if $gyr[x,y](z)\in V(r)$, which implies $f(gyr[x,y](z))=f(z)$. For every $z\in G$, if $z\in V(r)$, since $G$ is a strongly topological gyrogroup and $gyr[x, y]$ is an automorphism, we have that $gyr[x, y](V(r))=V(r)$. It follows that $gyr[x, y](z)\in V(r)$.
On the contrary, we assume that $gyr[x,y](z)\in V(r)$. Suppose that $z\not\in V(r)$. However, since $gyr[x,y](V(r))=V(r)$ and $gyr[x,y]$ is an automorphism, $gyr[x,y](z)\not\in V(r)$, which is a contradiction.

Now we can prove that $N$ is a prenorm as follows.

\smallskip
(PN1) Obviously, $N(0)=0$.

\smallskip
(PN2) $N(x\oplus y)\leq N(x)+N(y)$ for any $x, y\in G$.

Indeed, take an arbitrary $x, y\in G$. Then we have $N(x\oplus y)=\sup_{z\in G}|f((x\oplus y)\oplus z)-f(z)|$. Since $f((x\oplus y)\oplus z)=f(x\oplus (y\oplus gyr[y,x](z)))$, we have that
\begin{eqnarray}
f((x\oplus y)\oplus z)-f(z)&=&f(x\oplus (y\oplus gyr[y, x](z)))-f(y\oplus gyr[y, x](z))\nonumber\\
& &+f(y\oplus gyr[y, x](z))-f(gyr[y, x](z))+f(gyr[y, x](z))-f(z).\nonumber
\end{eqnarray}
Therefore, $N(x\oplus y)\leq N(x)+N(y)+N(f(gyr[y,x](z))-f(z))=N(x)+N(y)$ since $f(gyr[y,x](z))=f(z)$.

\smallskip
(PN3) $N(\ominus x)=N(x)$ for any $x\in G$.

Indeed, $N(\ominus x)=\sup_{y\in G}|f((\ominus x)\oplus y)-f(y)|=\sup_{y\in G}|f((\ominus x)\oplus y)-f(x\oplus ((\ominus x)\oplus y))|=N(x)$.

From the above, we obtain that $N$ is a prenorm. Next we shall$$\{x\in G:N(x)<1/2^{n}\}\subset U_{n}\subset\{x\in G:N(x)\leq 2/2^{n}\}$$ for any $n\in\omega$.

It is clear that $f(0)=0$. Assume that $N(x)<1/2^{n}$, for some $x\in G$. Then $f(x)=|f(0\oplus x)-f(0)|\leq N(x)<1/2^{n}$, which implies that $x\in V(1/2^{n})=U_{n}$. This proves that $\{x\in G:N(x)<1/2^{n}\}\subset U_{n}$.

Let $x$ be any point of $V(1/2^{n})$. Obviously, for any point $y\in G$, there exists a positive integer $k$ such that $(k-1)/2^{n}\leq f(y)<k/2^{n}$. Then $y\in V(k/2^{n})$. Since $x\in V(1/2^{n})$ and $\ominus x\in V(1/2^{n})$, it follows that $x\oplus y$ and $(\ominus x)\oplus y$ are in $V(1/2{n})\oplus V(k/2^{n})\subset V((k+1)/2^{n})$. Therefore, $f(x\oplus y)\leq (k+1)/2^{n}$ and $f((\ominus x)\oplus y)\leq (k+1)/2^{n}$. From this and the inequality $(k-1)/2^{n}\leq f(y)$, we obtain: $f(x\oplus y)-f(y)\leq 2/2^{n}$ and $f((\ominus x)\oplus y)-f(y)\leq 2/2^{n}$. Substituting $x\oplus y$ for $y$ in the last inequality, we get: $f(y)-f(x\oplus y)\leq 2/2^{n}$. Hence, we obtain that $|f(x\oplus y)-f(y)|\leq 2/2^{n}$, for each $y\in G$. Therefore, $N(x)\leq 2/2^{n}$.
\end{proof}

\begin{lemma}\label{h}
Let $(G,\tau ,\oplus)$ be a strongly topological gyrogroup. If $H$ is a compact subgyrogroup which has a countable character in $G$, then there exists a compact $L$-subgyrogroup $P\subset H$ such that $P$ has a countable character in $G$ and the quotient space $G/P$ is metrizable.
\end{lemma}

\begin{proof}
Let $\mathscr{U}$ be a symmetric neighborhood base at $0$ such that $gyr[x, y](U)=U$ for any $x, y\in G$ and $U\in\mathscr{U}$. Suppose that $\{U_{n}:n\in \omega \}$ is a countable base for $G$ at $H$. We define by induction a sequence $\{V_{n}: n\in \omega\}$ of open symmetric neighborhoods of the identity $0$ in $G$ such that $V_{n}\in\mathscr{U}$ and $V_{n+1}\oplus V_{n+1}\subset V_{n}\cap U_{n}$ for each $n\in \omega$. Put $P=\bigcap _{n\in \omega}V_{n}$. Then, by Lemma \ref{b}, $P$ is a compact $L$-subgyrogroup of $G$, $P\subset H$, and $\{V_{n}: n\in \omega \}$ is a countable base of $G$ at $P$. Let $\pi _{P}: G\rightarrow G/P$ be the quotient mapping of $G$ onto the left coset space $G/P$.

Apply Lemma \ref{s} to choose a continuous prenorm $N$ on $G$ which satisfies $$N(gyr[x,y](z))=N(z)$$ for any $x, y, z\in G$ and $$\{x\in G: N(x)<1/2^{n}\}\subset V_{n}\subset \{x\in G: N(x)\leq 2/2^{n}\},$$ for each integer $n\geq 0$.
It is clear that $N(x)=0$ if and only if $x\in P$.

We claim that $N(x\oplus p)=N(x)$ for every $x\in G$ and $p\in P$. Indeed, for every $x\in G$ and $p\in P$, $N(x\oplus p)\leq N(x)+N(p)=N(x)+0=N(x)$.  Moreover, by the definition of $N$, we observe that $N(gyr[x, y](z))=N(z)$ for every $x, y, z\in G$. Since $H$ is a $L$-subgyrogroup, it follows from Lemma~\ref{a} that
\begin{eqnarray}
N(x)&=&N((x\oplus p)\oplus gyr[x,p](\ominus p))\nonumber\\
&\leq&N(x\oplus p)+N(gyr[x,p](\ominus p))\nonumber\\
&=&N(x\oplus p)+N(\ominus p)\nonumber\\
&=&N(x\oplus p).\nonumber
\end{eqnarray}
Therefore, $N(x\oplus p)=N(x)$ for every $x\in G$ and $p\in P$.

Now define a function $d$ from $G\times G$ to $\mathbb{R}$ by $d(x,y)=|N(x)-N(y)|$ for all $x,y\in G$. Obviously, $d$ is continuous. We show that $d$ is a pseudometric.

\smallskip
(1) For any $x, y\in G$, if $x=y$, then $d(x, y)=|N(0)-N(0)|=0$.

\smallskip
(2) For any $x, y\in G$, $d(y, x)=|N(y)-N(x)|=|N(x)-N(y)|=d(x, y)$.

\smallskip
(3) For any $x, y, z\in G$, we have
\begin{eqnarray}
d(x, y)&=&|N(x)-N(y)|\nonumber\\
&=&|N(x)-N(z)+N(z)-N(y)|\nonumber\\
&\leq&|N(x)-N(z)|+|N(z)-N(y)|\nonumber\\
&=&d(x, z)+d(z, y).\nonumber
\end{eqnarray}

If $x'\in x\oplus P$ and $y'\in y\oplus P$, there exist $p_{1},p_{2}\in P$ such that $x'=x\oplus p_{1}$ and $y'=y\oplus p_{2}$, then $$d(x', y')=|N(x\oplus p_{1})-N(y\oplus p_{2})|=|N(x)-N(y)|=d(x, y).$$ This enables us to define a function $\varrho $ on $G/P\times G/P$ by $$\varrho (\pi _{p}(x),\pi _{p}(y))=d(\ominus x\oplus y, 0)+d(\ominus y\oplus x, 0)$$ for any $x, y\in G$.

It is obvious that $\varrho $ is continuous, and we verify that $\varrho $ is a metric on $Y=G/P$.

\smallskip
(1) Obviously, for any $x, y\in G$, then
\begin{eqnarray}
\varrho (\pi _{P}(x),\pi _{P}(y))=0&\Leftrightarrow&d(\ominus x\oplus y, 0)=d(\ominus y\oplus x, 0)=0\nonumber\\
&\Leftrightarrow&N(\ominus x\oplus y)=N(\ominus y\oplus x)=0\nonumber\\
&\Leftrightarrow&\ominus x\oplus y\in P\ \mbox{and}\ \ominus y\oplus x\in P\nonumber\\
&\Leftrightarrow&y\in x+P\ \mbox{and}\ x\in y+P\nonumber\\
&\Leftrightarrow&\pi _{P}(x)=\pi _{P}(y).\nonumber
\end{eqnarray}

\smallskip
(2) For every $x,y\in G$, it is obvious that $\varrho (\pi _{P}(y), \pi _{P}(x))=\varrho (\pi _{P}(x),\pi _{P}(y))$.

\smallskip
(3) For every $x, y, z\in G$, it follows from \cite[Theorem 2.11]{UA2005} that
\begin{eqnarray}
\varrho (\pi _{P}(x),\pi _{P}(y))&=&N(\ominus x\oplus y)+N(\ominus y\oplus x)\nonumber\\
&=&N((\ominus x\oplus z)\oplus gyr[\ominus x,z](\ominus z\oplus y))\nonumber\\
&&+N((\ominus y\oplus z)\oplus gyr[\ominus y,z](\ominus z\oplus x))\nonumber\\
&\leq&N(\ominus x\oplus z)+N(gyr[\ominus x,z](\ominus z\oplus y))\nonumber\\
&&+N(\ominus y\oplus z)+N(gyr[\ominus y,z](\ominus z\oplus x))\nonumber\\
&=&N(\ominus x\oplus z)+N(\ominus z\oplus y)+N(\ominus y\oplus z)+N(\ominus z\oplus x)\nonumber\\
&=&d(\ominus x\oplus z, 0)+d(\ominus z\oplus x, 0)+d(\ominus z\oplus y, 0)+d(\ominus y\oplus z, 0)\nonumber\\
&=&\varrho (\pi _{P}(x),\pi _{P}(z))+\varrho (\pi _{P}(z),\pi _{P}(y)).\nonumber
\end{eqnarray}

Let us verify that $\varrho$ generates the quotient topology of the space $Y$. Given any points $x\in G$, $y\in Y$ and any $\varepsilon>0$, we define open balls, $$B(x, \varepsilon)=\{x'\in G: d(x',x)<\varepsilon\}$$ and $$B^{*}(y, \varepsilon)=\{y'\in G/P: \varrho (y',y)<\varepsilon\}$$ in $X$ and $Y$, respectively. Obviously, if $x\in G$ and $y=\pi _{P}(x)$, then we have $B(x, \varepsilon)=\pi ^{-1}_{P}(B^{*}(y, \varepsilon))$. Therefore, the topology generated by $\varrho$ on $Y$ is coarser than the quotient topology. Suppose that the preimage $O=\pi ^{-1}_{P}(W)$ is open in $G$, where $W$ is a non-empty subset of $Y$. For every $y\in W$, we have $\pi ^{-1}_{P}(y)=x\oplus P\subset O$, where $x$ is an arbitrary point of the fiber $\pi ^{-1}_{P}(y)$. Since $\{V_{n}: n\in \omega\}$ is a base for $G$ at $P$, there exists $n\in \omega$ such that $x\oplus V_{n}\subset O$. Then there exists $\delta>0$ such that $B(x, \delta)\subset x\oplus V_{n}$. Therefore, we have $\pi ^{-1}_{P}(B^{*}(y, \delta))=B(x, \delta)\subset x\oplus V_{n}\subset O$. It follows that $B^{*}(y, \delta)\subset W$. So the set $W$ is the union of a family of open balls in $(Y, \varrho)$. Hence, $W$ is open in $(Y, \varrho)$, which proves that the metric and quotient topologies on $Y=G/P$ coincide.
\end{proof}

The following theorem gives a characterization for a strongly topological gyrogroup being a feathered space.

\begin{theorem}\label{i}
A strongly topological gyrogroup $(G, \tau, \oplus)$ is feathered if and only if it contains a compact $L$-subgyrogroup $H$ such that the quotient space $G/H$ is metrizable.
\end{theorem}

\begin{proof}
Let $\varphi$ be the natural homomorphism from $G$ to $G/H$. Suppose that $G$ contains a compact $L$-subgyrogroup $H$ such that the quotient space $G/H$ is metrizable. Then, by Theorem \ref{f}, the mapping $\varphi: G\rightarrow G/H$ is perfect. We now claim that the subgyrogroup $H\subset G$ has countable character in $G$. Indeed, let $\{U_{n}:n\in \omega\}$ be a countable base for $G/H$ at the point $\varphi(0)$. Then the family $\{V_{n}:n\in \omega\}$ is a base for $G$ at $H$, where $V_{n}=\varphi^{-1}(U_{n})$ for each $n\in \omega$. Indeed, for every open neighborhood $O$ of $H$ in $G$, $F=G\backslash O$ is closed in $G$ and the image $\varphi(F)$ is a closed subset of $G/H$ which does not contain $\varphi(0)$. Hence, $U_{n}\bigcap \varphi(F)=\emptyset$ for some $n\in \omega$, which in turn implies that $V_{n}\bigcap F=\emptyset$. Therefore, $H\subset V_{n}\subset O$. Therefore, $\chi (H, G)\leq \omega$. Since $H$ is compact, the gyrogroup $G$ is feathered.

Conversely, if the strongly topological gyrogroup $G$ is feathered, then it contains a compact $L$-subgyrogroup $H$ of countable character in $G$ by Theorem \ref{g}. Hence it follows from Theorem \ref{h} that the quotient space $G/H$ is metrizable.
\end{proof}

The following two corollaries give partial answers to Questions~\ref{kk} and~\ref{jj}, respectively.

\begin{corollary}\label{l}
Every feathered strongly topological gyrogroup is paracompact.
\end{corollary}

\begin{proof}
Let $(G, \tau, \oplus)$ be a feathered strongly topological gyrogroup. By Theorem \ref{i}, $G$ contains a compact $L$-subgyrogroup $H$ such that the left coset space $G/H$ is metrizable. Moreover, it follows from Theorem~\ref{f} that the quotient mapping $\varphi: G\rightarrow G/H$ is perfect. Then since paracompactness is an inverse invariant under perfect mappings, it follows that $G$ is paracompact.
\end{proof}

\begin{corollary}
Every feathered strongly topological gyrogroup is a $D$-space.
\end{corollary}

A Tychonoff space $X$ is called $\check{C}$ech-complete if $X$ is homeomorphic to a $G_{\delta}$-set in a compact space. The class of $\check{C}$ech-complete spaces is stable with respect to taking closed subspaces, countable products and perfect images \cite{E}. A $\check{C}$ech-complete gyrogroup is a topological gyrogroup whose underlying space is $\check{C}$ech-complete. Moreover, all $\check{C}$ech-complete gyrogroups are feathered.

\begin{theorem}
A strongly topological gyrogroup $(G, \tau, \oplus)$ is $\check{C}$ech-complete if and only if it contains a compact $L$-subgyrogroup $H$ such that the quotient space $G/H$ is metrizable by a complete metric.
\end{theorem}

\begin{proof}
Suppose that the strongly topological gyrogroup $G$ is $\check{C}$ech-complete. Then $G$ is feathered, thus Theorem \ref{g} implies that $G$ contains a compact $L$-subgyrogroup $H$ of countable character in $G$. Furthermore, the quotient mapping $\varphi: G\rightarrow G/H$ is perfect by Theorem~\ref{f}, then it follows from Theorem \ref{h} that the left coset space $G/H$ is metrizable. Since $G$ is a $\check{C}$ech-complete space, it follows from \cite[Theorem 3.9.10]{E} that $G/H$ is $\check{C}$ech-complete, then $G/H$ is completely metrizable by \cite[Theorem 4.3.26]{E}.

Conversely, if the strongly topological gyrogroup $G$ contains a compact $L$-subgyrogroup $H$ such that the left coset space $G/H$ is metrizable by a complete metric, then $G/H$ is $\check{C}$ech-complete by \cite[Theorem 4.3.26]{E}. Since the quotient mapping $\varphi: G\rightarrow G/H$ is perfect, and $\check{C}$ech-completeness is an inverse invariant under perfect mappings by \cite[Theorem 3.9.10]{E}, we have shown that the space $G$ is $\check{C}$ech-complete.
\end{proof}

\begin{corollary}
Every $\check{C}$ech-complete strongly topological gyrogroup is paracompact.
\end{corollary}

\smallskip
\section{NSS-gyrogroups and $P$-gyrogroups}
In this section, we mainly discuss the NSS-gyrogroups and $P$-gyrogroups, and prove that every locally compact $NSS$-gyrogroup is first-countable and every Lindel$\ddot{o}$f $P$-gyrogroup is Ra$\breve{\i}$kov complete. Let $(G, \tau, \oplus)$ be a topological gyrogroup. We say that $G$ is a gyrogroup with {\it no small subgyrogroups} or, for brevity, an $NSS$-gyrogroup if there exists a neighborhood $V$ of the neutral element $0$ such that every subgyrogroup $H$ of $G$ contained in $V$ is trivial, that is, $H=\{0\}$.

It is well known that each topological group with a countable pseudocharacter is submetrizable. Therefore, we have the question.

\begin{question}
Let $(G, \tau, \oplus)$ be a topological gyrogroup with a countable pseudocharacter. Is $G$ submetrizable?
\end{question}

This question is still unknown for us. However, we have Theorem~\ref{x}. First, we prove a lemma.

\begin{lemma}\label{x2}
Let $(G, \tau, \oplus)$ be a topological gyrogroup, and let $F$ be a non-empty compact $G_{\delta}$-set in $G$. Then there exists a subgyrogroup $P$ such that $P$ is a $G_{\delta}$-set in $G$, $0\in P$ and $F\oplus P\subset F$.
\end{lemma}

\begin{proof}
Let $F=\bigcap_{n=1}^{\infty} U_{n}$, where $U_{n}$ is open and $U_{n+1}\subset U_{n}$ for each $n\in\mathbb{N}$. By induction, it follows from \cite[Lemma 2.3]{TL} that there exists a sequence $\{V_{n}: n\in\mathbb{N}\}$ of open neighborhoods of $0$ such that the following conditions hold:

(1) $F\oplus V_{n}\subset U_{n}$ for each $n\in\mathbb{N}$;

(2) $\ominus V_{n}=V_{n}$ for each $n\in\mathbb{N}$;

(3) $V_{n+1}\oplus V_{n+1}\subset V_{n}$ for each $n\in\mathbb{N}$.

Let $P=\bigcap_{n=1}^{\infty}V_{n}$. Obviously, $0\in P$, $P$ is a $G_{\delta}$-set and $F\oplus P\subset F$. It suffices to prove that $P$ is a closed subgyrogroup. By Theorem~\ref{t29}, $P$ is closed by (3) above. We only prove that $P$ is a subgyrogroup. Indeed, take arbitrary $x, y\in P$. Then $x, y\in V_{n}$ for each $n\in \mathbb{N}$, then $x\oplus y\in V_{n}$ and $\ominus x\in V_{n}$ since $V_{n+1}\oplus V_{n+1}\subset V_{n}$ and $\ominus V_{n}=V_{n}$. Therefore, $x\oplus y\in P$ and $\ominus x\in P$. By the arbitrary, we see that $P$ is a subgyrogroup.
\end{proof}

\begin{theorem}\label{x}
If a topological gyrogroup $(G, \tau, \oplus)$ is an $NSS$-gyrogroup, then the following two conditions are equivalent:

(1) there exists a non-empty compact $G_{\delta}$-set $F$ in $G$;

(2) the identity element $0$ of $G$ is a $G_{\delta}$-point in $G$.
\end{theorem}

\begin{proof}
It is clear that the second condition implies the first one. Let us show that the first condition implies the second one. Since the space $G$ is homogeneous, we can assume that $0\in F$. By the regularity of $G$, we can also assume that $F\subset U$, where $U$ is an open neighborhood of $0$ such that all subgyrogroups of $G$ contained in $U$ are trivial. By Lemma $\ref{x2}$, there exists a $G_{\delta}$-set $P$ in $G$ which is a subgyrogroup of $G$ such that $0\in P\subset F$. Since $F\subset U$, it follows that $P\subset U$. Therefore, $P=\{0\}$ since $G$ is a $NSS$-gyrogroup. Since $P$ is a $G_{\delta}$-set in the gyrogroup $G$, we are done.
\end{proof}

\begin{theorem}
Every locally compact $NSS$-gyrogroup $(G, \tau, \oplus)$ is first-countable.
\end{theorem}

\begin{proof}
It is clear that every locally compact regular space contains a non-empty compact $G_{\delta}$-set. Therefore, by Theorem $\ref{x}$, the identity element $0$ of $G$ is a $G_{\delta}$-point in $G$. Since the space $G$ is locally compact and Hausdorff, it follows that $G$ is first-countable at $0$. Hence, by homogeneity, the space $G$ is first-countable.
\end{proof}

Finally, we discuss $P$-gyrogroups, and prove that every Lindel$\ddot{o}$f $P$-gyrogroup $(G, \tau, \oplus)$ is Ra$\breve{\i}$kov complete. Recall that $X$ is a $P$-space if every $G_{\delta}$-set in $X$ is open. Similarly, a $P$-gyrogroup is a topological gyrogroup whose underlying space is a $P$-space. It is well known that every regular $P$-space is zero-dimensional.

\begin{lemma}\label{r}
Suppose that $(G, \tau, \oplus)$ is a $P$-gyrogroup. Then:

\smallskip
(a) $G$ has a base at the identity element $0$ consisting of open subgyrogroups.

\smallskip
(b) If $H$ is a $L$-subgyrogroup of $G$, then the coset space $G/H$ is also a $P$-space.

\smallskip
(c) If $G$ is a dense subgyrogroup of a topological gyrogroup $H$, then $H$ is a $P$-gyrogroup.
\end{lemma}

\begin{proof}
($a$) It suffices to prove that for each neighborhood $U$ of $0$ there exists an open subgyrogroup $V$ such that $V\subset U$. Indeed, take an arbitrary neighborhood $U$ of the identity element $0$ in $G$. Then there exists a sequence $\{U_{n}:n\in \omega\}$ of open symmetric neighborhoods of $0$ in $G$ such that $U_{0}\subset U$ and $U_{n+1}\oplus U_{n+1}\subset U_{n}$ for each $n\in \omega$. Then $V= \bigcap _{n=0}^{\infty}U_{n}$ is a subgyrogroup of $G$ lying in $U$. Since $G$ is a $P$-gyrogroup, $V$ is open in $G$. Every open subgyrogroup of $G$ is closed by \cite[Proposition 7]{AW}, so $G$ has a base of closed and open sets at the identity element $0$.

(b) By Theorem~\ref{t00000}, the natural homomorphism $\varphi: G\rightarrow G/H$ is a continuous open. Assume that $Q$ is a $G_{\delta}$-set in $G/H$, then $\varphi^{-1}(Q)$ is a $G_{\delta}$-set in the $P$-gyrogroup $G$, so that $\varphi^{-1}(Q)$ is open in $G$. Therefore, $Q=\varphi(\varphi^{-1}(Q))$ is open in $G/H$. This proves that $G/H$ is a $P$-space.

(c) Let $U=\bigcap_{n\in\mathbb{N}}U_{n}$, where $\{U_{n}:n\in \omega\}$ is a sequence of open neighborhoods of the identity element $0$ in $H$. Next we prove that $U$ is open in $H$. Indeed, there exists a sequence $\{V_{n}:n\in \omega\}$ of open symmetric neighborhoods of $H$ such that $V_{n+1}\oplus V_{n+1}\subset V_{n}\subset U_{n}$ for each $n\in \omega$. Then $N=\bigcap_{n=0}^{\infty}V_{n}$ is a closed subgyrogroup of $H$ by Theorem~\ref{ttt}. Since $G$ is a $P$-gyrogroup, $P=N\cap G=\bigcap _{n=0}^{\infty}(V_{n}\cap G)$ is a clopen subgyrogroup of $G$. Therefore, we can take an open set $W$ in $H$ such that $W\cap G=P$. It is clear that $0\in W$, and the density of $G$ in $H$ implies that $cl_{H}W=cl_{H}P\subset N$. Since $N\subset U=\bigcap _{n=0}^{\infty}U_{n}$, we conclude that $U$ contains the open neighborhood $W$ of $0$ in $H$. Therefore, $H$ is a $P$-gyrogroup.
\end{proof}

\begin{lemma}
Let $(G, \tau, \oplus)$ be an $\omega$-narrow $P$-gyrogroup. Then every gyrogroup homomorphic continuous image $K$ of $G$ with $\psi (K)\leq \omega$ is countable.
\end{lemma}

\begin{proof}
Consider a continuous homomorphism $\pi :G\rightarrow K$ onto a gyrogroup $K$ with countable pseudocharacter. Since $G$ is a $P$-gyrogroup, the kernel $N$ of $\pi$ is an open subgyrogroup of $G$. By assumption, the gyrogroup $G$ is $\omega$-narrow, so it can be covered by countably many gyrotranslations of $N$. Therefore, the quotient gyrogroup $G/N$ is countable. Finally, it follows from \cite[Theorem 28]{ST} that the gyrogroups $K$ and $G/N$ are algebraically isomorphic, and it follows that $|K|\leq \omega$.
\end{proof}

\begin{lemma}\cite[Lemma 4.4.3]{AA}\label{q}
A Lindel\"{o}f subset $Y$ of a Hausdorff $P$-space $X$ is closed in $X$.
\end{lemma}

\begin{corollary}
Let $\pi: G\rightarrow H$ be a continuous onto gyrogroup homomorphism of Lindel\"{o}f $P$-gyrogroups. Then $\pi$ is open.
\end{corollary}

\begin{proof}
By Lemma \ref{q}, the homomorphism $\pi$ is closed and, hence, quotient. Let $U$ be an open subset of $G$. Then $\pi ^{-1}\{\pi (U)\}=U\oplus K$ is open in $G$, where $K$ is the kernel of $\pi$. Since $\pi$ is a quotient mapping, the set $\pi (U)$ has to be open in $H$. So $\pi$ is an open homomorphism.
\end{proof}

\begin{theorem}
Every Lindel\"{o}f $P$-gyrogroup $(G, \tau, \oplus)$ is Ra$\check{\imath}$kov complete.
\end{theorem}

\begin{proof}
Suppose that $G$ is a dense subgyrogroup of a topological gyrogroup $H$. By (c) of Lemma \ref{r}, $H$ is a $P$-gyrogroup. Then Lemma \ref{q} implies that $G$ is closed in $H$, whence it follows that $G=H$. Therefore, the gyrogroup $G$ is Ra$\check{\imath}$kov complete.
\end{proof}

{\bf Acknowledgements}. The authors are thankful to the
anonymous referees for valuable remarks and corrections and all other sort of help related to the content of this article.

\end{document}